\documentclass[11pt,a4paper]{amsart}
\usepackage{amsfonts, amsmath, amssymb, amscd, amsthm, array}

\usepackage{multicol}
\usepackage{subfig}
\usepackage{color}
\usepackage[all]{xy}
\usepackage{hyperref,url}

\makeatletter
\def\blfootnote{\xdef\@thefnmark{}\@footnotetext}
\makeatother

\newtheorem{thm}{Theorem}[section]
\newtheorem{cor}[thm]{Corollary}
\newtheorem{lemma}[thm]{Lemma}
\newtheorem{prop}[thm]{Proposition}

\newtheorem{conj}[thm]{Conjecture}
\theoremstyle{definition}
\newtheorem{df}[thm]{Definition}
\newtheorem{pri}[thm]{Meta-Conjecture}
\theoremstyle{remark}
\newtheorem{rem}[thm]{Remark}

\newcommand{\cE}{\mathcal{E}}

\newcommand{\cN}{\mathcal{N}}
\newcommand{\e}{\varepsilon }
\newcommand{\N}{\mathbb{N}}
\newcommand{\Z}{{\mathbb Z}}

\newcommand{\dc}{\operatorname{dc}}
\newcommand{\ds}{\operatorname{ds}}
\newcommand{\supp}{\operatorname{Supp}}
\newcommand{\ball}{\mathbb{B}}

\begin{document}

\title[Degree of commutativity of infinite groups]{Degree of commutativity of infinite
groups}

\address{Vanderbilt University, Department of Mathematics, 1326 Stevenson Center,
Nashville, TN 37240, USA.}
\author{Yago Antol\'{i}n}
\email[Yago Antol\'{i}n]{yago.anpi@gmail.com}

\address{Department of Mathematics, University of Southampton, Highfield, Southampton
SO17 1BJ, United Kindgom.}
\author{Armando Martino}
\email[Armando Martino]{A.Martino@soton.ac.uk}

\address{Departament de Matem\`atiques, EPSEM, Universitat Polit\`{e}cnica de Catalunya,
Av. Bases de Manresa 61--73, 08242-Manresa, Barcelona (Catalonia).}
\author{Enric Ventura}
\email[Enric Ventura]{enric.ventura@upc.edu}


\date{\today}

\begin{abstract}
We prove that, in a finitely generated residually finite group of subexponential
growth, the proportion of commuting pairs is positive if and only if the group is
virtually abelian. In particular, this covers the case where the group has polynomial
growth (i.e., virtually nilpotent groups, where the hypothesis of residual finiteness
is always satisfied). We also show that, for non-elementary hyperbolic groups, the
proportion of commuting pairs is always zero.
\end{abstract}


\maketitle


\section{Introduction}

In a finite group $G$, a way of studying a property one is interested in is by
counting, or estimating, the probability that this property holds among the elements of
$G$. For example, we can look at the proportion of pairs of elements in $G$ which
commute to each other,
 $$
\dc(G)=\frac{|\{(u,v)\in G^{\, 2} \ : \ uv=vu \}|}{|G|^2},
 $$
and call it the \emph{degree of commutativity} of $G$. Of course, $\dc(G)$ is a
rational number between 0 and 1, $\dc(G)=1$ if and only if $G$ is abelian, and the
closer $\dc(G)$ is to 1 the ``more abelian" $G$ will be. An interesting result due to
Gustafson~\cite{G} states that there is no finite group with $\dc(G)$ strictly between
$5/8$ and $1$, i.e.,

\begin{thm}[Gustafson, \cite{G}]\label{thm:5/8}
Let $G$ be a finite group. If $\dc(G)>5/8$ then $G$ is abelian.
\end{thm}

Although at a first look this result could seem somehow surprising, the proof is very
elementary and straightforward. The bound in Gustafson's result is tight, since easy
computations show that the degree of commutativity of the quaternion group $Q_8$ is,
precisely, $\dc(Q_8)=5/8$.

This implies that the range of the map $\dc$ from finite groups to $[0,1]$ does not
cover the whole set of rational points $\mathbb{Q}\cap [0,1]$, but leaves (infinitely
many!) gaps on it; also, clearly, $\dc(G)\neq 0$ for every finite group $G$. Several
interesting conjectures have been established about this set, and some results have
been proven, see Hegarty~\cite{He} and the references therein. This is in stunning
contrast with the situation for semigroups: Ponomarenko--Selinski showed in~\cite{Po}
that, for \emph{every} rational number $q\in (0,1]$ there exists a semigroup $S$ such
that $\dc(S)=q$ (with the obvious analogous definition of $\dc$ for semigroups).

The aim of this paper is to generalize this result to the context of infinite groups.
Of course, in an infinite ambient the first obvious difficulty is the meaning of
``proportion of elements". For compact groups,  Gustafson \cite{G} himself proved a version of Theorem \ref{thm:5/8} using the Haar measure.
In general, considering only one measure might not be enough and we
 propose to deal with this by considering in $G$ a
sequence of probability measures $\{\mu_n\}_{n\in \N}$.
With them, we can compute the probability (with respect to
$\mu_n^{\times 2}$, the product measure in $G^{\, 2}=G\times G$) that two elements in
\emph{$\supp(\mu_n)$} commute, and then take the limit when $n$ tends to infinity. A
first technical problem is the existence of this limit (to begin with, we will take the
limsup); a second problem is that, in principle, all depends on the choice of the
sequence of measures $\{\mu_n\}_{n\in \N}$, i.e., on the \emph{direction} taken
\emph{to infinity}.

To state the main definitions we are going to adopt a more general point of view, and
will define the notion of satisfiability of an arbitrary set of equations (not just of
the commuting equation $xy=yx$). Let $k\geqslant 1$ and $F_k =F(\{x_1, \ldots ,x_k\})$
be the (rank $k$) free group on the alphabet $\{x_1, \ldots ,x_k\}$. An element $w\in
F_k$ will be called an \emph{equation in $k$   variables} (thinking of $w(x_1,\ldots ,x_k
)=1$); and a $k$-tuple of elements in $G$, say $(g_1, \ldots ,g_k)\in G^{\, k}=G\times
\cdots \times G$, is a \emph{solution to $w$ in $G$} if $w(g_1, \ldots ,g_k)=_G 1$.
Analogously, a (possibly infinite) set of elements $\cE \subseteq F_k$ is a
\emph{system of equations in $k$variables}, and we define solutions in a group $G$ in
the obvious way. For example, a pair of elements in a group $G$, say $(g_1, g_2)\in
G^{\, 2}$, form a solution to the equation $[x_1,x_2]=1$ if and only if they commute in
$G$.

A word $w\in F_k$ is said to be a \emph{law} in a group $G$ if every $k$-tuple in
$G^{\, k}$ is a solution to $w$. For example, $[x_1,x_2]$ is a law in $G$ if and only
if $G$ is abelian; or $[[x_1,x_2],[x_3,x_4]]$ is a law in $G$ if and only if $G$ is
metabelian. Also, we say that $w$ is a \emph{virtual law} in $G$ if $w$ is a law in a
certain finite index subgroup $H\leqslant G$. Analogous definitions apply to arbitrary
subsets $\cE\subseteq F_k$.

\begin{df}\label{gendef}
Let $G$ be a group, let $\{\mu_n\}_{n\in \N}$ be a sequence of probability measures on $G$, and let $\cE \subseteq F_k$ be a set of equations in $k$ variables. We define the {\it degree of satisfiability of $\cE$ in $G$ with respect to $\{\mu_n\}_{n\in \N}$} as
 $$
\ds(G,\cE,\{\mu_n\}_{n\in \N})=\limsup_{n\to \infty} \mu_n^{\times k}\big( \{(g_1,\ldots ,g_k)\in
G^{\, k} \ : \  (g_1,\ldots ,g_k) \text{ is a solution of } \cE\}\big),
$$
where $\mu_n^{\times k}$ is the product measure in $G^{\, k}$. For the particular case
$k=2$ and $\cE=\{[x_1,x_2]\}$, we denote $\ds(G,\{[x_1,x_2]\},\{\mu_n\}_{n\in \N})$
simply as $\dc(G, \{\mu_n\}_{n\in \N})$ and refer to it as the \emph{degree of
commutativity of $G$ with respect to $\{\mu_n\}_{n\in \N}$}.
\end{df}

Suppose now that $G$ is virtually abelian, i.e., it has an abelian subgroup $H\leqslant
G$ of finite index, say $d=[G:H]$. Roughly speaking, one of each $d^2$ pairs of
elements $(g_1, g_2)\in G^{\, 2}$ belongs to $H^2$ and so commute. Hence, intuitively,
the degree of commutativity of $G$ 
should be about $1/d^2>0$, or more. Let us accept that the actual value of this limsup
is not very significant due to the subtleties of the measures $\mu_n$; but even with
this fact, it seems plausible to think that ``reasonable" sequences of probability
measures $\{\mu_n\}_{n\in \N}$ on $G$ will keep the mass of $H$ visible in $G$ (as
approximately $(1/d)$-th part of the total), and so will still satisfy $\dc(G,
\{\mu_n\}_{n\in \N})>0$.

Suppose now that $G$ is a group and $\{\mu_n\}_{n\in \N}$ a sequence of probability
measures on it, in such a way that $\dc(G, \{\mu_n\}_{n\in \N})>0$. It is interesting
to ask whether this visible amount of commutativity in $G$ can be organized in an
algebraic way into forming, for example, an abelian subgroup of finite index; this
would say that $G$ is virtually abelian, providing an interesting convers to the
argumentation line in the previous paragraph. We tend to think that this is
philosophically the case, even for arbitrary systems of equations (maybe under some
additional technical assumptions), and formulate the following

\begin{pri}\label{meta}
Let $G$ be a group, let $\{\mu_n\}_{n\in \N}$ be a sequence of ``reasonable"  measures on $G$, and let $\cE \subseteq F_k$ be a set of
equations in $k$ variables. Then,
 $$
\ds(G,\cE,\{\mu_n\}_{n\in \N})>0 \quad \Longleftrightarrow \quad \cE \text{ is a virtual law
in } G.
 $$
In particular,
 $$
\dc(G,\{\mu_n\}_{n\in \N})>0 \quad \Longleftrightarrow \quad G \text{ is virtually abelian}.
 $$
\end{pri}

The main results in this paper follow this philosophy and are clear special cases of
this meta-conjecture, see Theorems~\ref{thm:main} and~\ref{thm:nafs} below.

We will make no attempt to give a precise definition of what ``reasonable" should mean
here. Just to mention some vague idea, it seems clear one should impose conditions on
these $\mu_n$'s forcing them to tend to cover the whole $G$ with enough homogeneity.
For example, if all the supports are finite (case in which the computations will be of combinatorial nature), it seems natural to ask them to form an ascending sequence $\supp(\mu_1)\subseteq \supp(\mu_2)\subseteq \cdots$ eventually covering the whole group $G$, i.e., $G=\cup_{n\in \N} \supp(\mu_n)$.

We will fix particularly easy and natural sequences of such probability measures (natural from the algebraic point of view) and will work with them, proving theorems in the spirit of the meta-conjecture. The following is one of the most natural ways of doing this. Assume $G$ is finitely generated, fix a finite set of generators $X$, consider the \emph{ball of radius $n$} centered at $1$ on the Cayley graph $\Gamma(G,X)$ of $G$ with respect to $X$, denoted $\ball_X(n)$, and choose $\mu_n$ to be the \emph{uniform measure} on $\ball_X(n)$ (i.e., $\mu_n$ assigns probability $1/|\ball_X(n)|$ to each element in $\ball_X(n)$, and 0 to all other elements in $G$). Clearly, this makes sense because $|\ball_X(n)|<\infty$, and satisfies the natural condition imposed above: $\{ 1\}=\ball_X(0)\subseteq \ball_X(1)\subseteq \ball_X(2)\subseteq \cdots$ and $G=\cup_{n\in \N} \ball_X(n)$.

Another interesting possibility consists on taking, into the $\ball_X(n)$, the
distribution associated to a \emph{simple random walk} of length $n$ in the Cayley
graph $\Gamma(G,X)$ (i.e., for $g\in G$, define $\mu_n(g)$ to be the probability that a
random walk of length $n$ in $\Gamma(G,X)$ starting at 1 ends at the element $g$; of
course, $\mu_n(g)=0$ if and only if $g\not\in \ball_X(n)$ and so, $\supp(\mu_n
)=\ball_X(n)$).

In the present paper we will consider only the uniform measure on balls, and will
concentrate our efforts to study the commutativity equation $\cE =\{ [x_1,x_2]\}$. We
wonder whether our results can be extended to other (systems of) equations, and other
probability measures on groups.

Particularizing Definition~\ref{gendef} to the commutativity case $\cE =\{
[x_1,x_2]\}$, and with respect to the uniform measure on $X$-balls, let us simplify
notation and names for the rest of the paper.

\begin{df} Let $G$ be a finitely generated group, and $X$ a finite generating set. The
\emph{degree of commutativity of $G$ with respect to $X$}, denoted $\dc_X(G)$, is
defined as
 $$
\dc_X(G) =\limsup_{n \to \infty} \frac{|\{(u,v)\in (\ball_X(n))^2 \ :
\ uv=vu \}| }{|\ball_X(n)|^2}.
 $$
\end{df}

We keep the subindex to specify the dependency of this notion, in general, from the set of generators, and to distinguish from the corresponding notion for finite groups. Of course, if $G$ is finite then $\dc_X(G)=\dc(G)$ for every generating set $X$ of $G$.

The main result in the paper is the following one, clearly in the spirit of~\ref{meta}.

\begin{thm}\label{thm:main}
Let $G$ be a finitely generated residually finite group of subexponential growth, and
let $X$ be a finite generating set for $G$. Then,
\begin{enumerate}
\item[(i)] $\dc_X(G)>0$ if and only if $G$ is virtually abelian;
\item[(ii)] $\dc_X(G)>5/8$ if and only if $G$ is abelian.
\end{enumerate}
\end{thm}

As a corollary we obtain that the positivity of $\dc_X(G)$ is independent from the
generating set $X$.

\begin{cor}\label{cor-main}
Let $G$ be a finitely generated residually finite group of subexponential growth, and
let $X$ and $Y$ be two finite generating sets for $G$. Then, $\dc_X(G)=0 \,
\Leftrightarrow \, \dc_Y(G)=0$.
\end{cor}

Note that, as a special case, Theorem~\ref{thm:main} and Corollary~\ref{cor-main} both
apply to polynomially growing groups: by Gromov's Theorem~\cite{Gromov-polyn_gr} these
are precisely the virtually nilpotent groups, and it is well known that all of them are
residually finite.


\begin{cor}
Let $G$ be a finitely generated group of polynomial growth, and let $X$ be a finite
generating set for $G$. Then: (i) $\dc_X(G)>0$ if and only if $G$ is virtually abelian; and (ii) $\dc_X(G)>5/8$ if and only if $G$ is abelian.
\end{cor}

We conjecture that same result is true without any hypothesis on the growth of $G$:

\begin{conj}\label{conj}
Let $G$ be a finitely generated group, and let $X$ be a finite generating set for $G$.
Then: (i) $\dc_X(G)>0$ if and only if $G$ is virtually abelian; and (ii) $\dc_X(G)>5/8$ if and only if $G$ is abelian.
\end{conj}

In view of Theorem~\ref{thm:main} (and since virtually abelian groups are
polynomially growing), one could prove Conjecture~\ref{conj} by showing that exponentially growing groups, and non-residually finite sub-exponentially growing groups $G$ all satisfy $\dc_X(G)=0$, for every finite generating set $X$. We have not been able to see this, but we can show the following particular (but interesting and significant) case:

\begin{thm}\label{thm:nafs}
Let $G$ be a  non-elementary hyperbolic group, and let $X$ be a finite generating set
for $G$. Then $\dc_X(G)=0$.
\end{thm}

\begin{rem}
All along the paper, all generating sets $X$ are assumed to generate $G$ as a group.
The same results remain true if one assumes that they generate $G$ as a monoid, modulo
changing the radii of some balls in the proof of \ref{thm:nafs}.
\end{rem}

\begin{rem}
There is a technical detail which is not yet completely clarified: is the limsup in the
definition of $\dc_X(G)$ always a real limit? Since they all are sequences of
nonnegative real numbers, whenever these limsups are equal to zero they automatically
are genuine limits (and this is the case for a vast majority of groups, as can be
deduced from Theorems~\ref{thm:main} and~\ref{thm:nafs}). However, further than that we
have not been able to prove that these limsups are always genuine limits, neither have
been able to construct an example where the limit does not exist. As far as we know,
this is an intriguing open question at the moment.
\end{rem}

\begin{rem}
Another intriguing detail is the dependency, or not, of the actual value of $\dc_X(G)$
from the set of generators. Corollary~\ref{cor-main} states that, under hypothesis on
$G$, the positivity of $\dc_X(G)$ does not depend on $X$. And if Conjecture~\ref{conj}
is true then the same will happen without any assumption on $G$. However, in the positive
realm, we have not been able to prove that the value of $\dc_X(G)$ is always
independent from $X$, neither have been able to construct a group $G$ with two finite
generating sets $X$ and $Y$ such that $\dc_X(G)\neq \dc_Y(G)$. As far as we know, this
is another intriguing open question at the moment.
\end{rem}

The generalization of the above questions to the general values $\ds(G,\cE, \{\mu_n
\}_{n\in \N})$ is not interesting: by constructing artificially enough probability measures $\{\mu_n\}_{n\in \N}$, concentrating their masses in nonsignificant parts of $G$, one could force the sequence in $\ds(G,\cE,\{\mu_n\}_{n\in \N})$ to be not convergent, and even to force $\ds(G,\cE,\{\mu_n\}_{n\in \N})$ to equal any real number in $[0,1]$. For example, take a group $G=\langle X\rangle$ with $\dc_X(G)=0$, fix an element $g\in G$ of infinite order, fix a non decreasing sequence $M_n$ of natural numbers, and consider $\mu_n$ to be the uniform distribution on the finite set $\supp{\mu_n }=\ball_X(n)\cup \{ g^r : |r|\leqslant M_n\}$. Taking $M_n=0$ for every $n$, we reproduce the fact $\dc_X(G)=0$, while taking the appropriate $M_n$'s we can force the limit $\dc(G,\{\mu_n\}_{n\in \N})$ to equal to any prefixed $s\in [0,1]$ (since all the powers of $g$ commute to each other); or even not to exist.

The paper is structured as follows. In Section~\ref{2} we consider polynomially growing
groups and prove several preliminary results in this case. In Section~\ref{3}, we
consider the sub-exponential case and prove our main result, Theorem~\ref{thm:main}.
Finally, in Section~\ref{4}, we look at the exponential case and prove
Theorem~\ref{thm:nafs} as a partial result in the direction of Conjecture~\ref{conj}.

\section{Groups of polynomial growth}\label{2}

Let $G$ be a finitely generated group, and $X$ a finite generating set for $G$.

The \emph{growth of $G$ with respect to $X$} is the non-decreasing function measuring
the size of the $X$-balls, $\gamma_X\colon \N\to \N$, $\gamma_X(n)=|\ball_X(n)|$. As it
is well known, this function depends on $X$, but its asymptotic growth, i.e., its
equivalence class modulo the following equivalence relation does not: two functions
$f,g\colon\N\to \N$ are \emph{equivalent} if and only if $f\preceq g\preceq f$, where
$f\preceq g$ if and only if there exist constants $C$ and $D$ such that $f(n)\leqslant
Cg(Dn)$ for all $n\geqslant 1$ (or, equivalently, for $n\gg 0$).

A group $G$ is said to have \emph{subexponential growth} if $\lim_{n\to \infty}
|\ball_X(n+1)|/|\ball_X(n)|=1$. And it is said to have \emph{polynomial growth} if
there exist constants $D$ and $d$ such that $|\ball_X(n)|\leq Dn^d$ for $n\gg 0$. It is
well known that being of subexponential/polynomial growth is independent from the
choice of generating set $X$. The celebrated Gromov's Theorem~\cite{Gromov-polyn_gr}
says that $G$ has polynomial growth if and only if it is virtually nilpotent. And,
precisely for (virtually) nilpotent groups $G$, a previous well-known result due to
Bass~\cite{Bass} states the existence of constants $C$ and $D$ such that $Cn^d\leqslant
|\ball_X(n)|\leqslant Dn^d$ for all $n\geqslant 1$, where $d$ is a natural number
depending only on $G$ (more precisely, $d=\sum_{i\geqslant 1} ir_i$, where $r_i$ is the
torsion-free rank of $G_i/G_{i+1}$ and $G=G_1\geqslant G_2 \geqslant \cdots \geqslant
1$ is the (finite) lower central series of $G$); hence, for polynomially growing groups
$G$, its growth function is exactly polynomial of a well defined degree $d$, up to
multiplicative constants. We then say that $G$ has polynomial growth \emph{of degree
$d$}. Finally, Grigorchuk found in~\cite{Gr} the first examples of groups of
\emph{intermediate growth}, i.e., groups $G$ of subexponential but not polynomial
growth.

A function $f\colon G\to \N$ is an \emph{estimate of the $X$-metric} if there exists a
constant $K>0$ such that
 \begin{equation}\label{estimate}
\frac{1}{K}f(u)\leqslant |u|_X \leqslant Kf(u),
 \end{equation}
for all $u\in G$. Here are two important examples of this notion: (i) if $Y$ is another
finite generating set for $G$, then $|\cdot|_Y$ is an estimate of the $X$-metric (and
$|\cdot|_X$ is an estimate of the $Y$-metric); and (ii) if $\langle Y\rangle=H\leqslant
G=\langle X\rangle$ is a non-distorted subgroup (i.e., there exists a constant $K>0$
such that, for every $h\in H$, $|h|_Y/K\leqslant |h|_X \leqslant K|h|_Y$), then
$|\cdot|_X$ restricted to $H$ is an estimate of the $Y$-metric for $H$.

For any function $f\colon G\to \N$, we define the \emph{$f$-ball of radius $n$} as the
pre-image under $f$ of $\{0,1,\ldots, n\}$, and we denote it by $\ball_f(n)$. Of
course, we have $\ball_f(0)\subseteq \ball_f(1)\subseteq \cdots$ and $G=\cup_{n\in \N}
\ball_f(n)$. Note also that, if $f$ is an estimation of the $X$-metric then, in
particular, $f$-balls are all  finite and almost all non-empty. In this case,
just as $\dc_X(G)$ denotes $\dc(G,\{\mu_n\})$, where $\mu_n$ is the uniform measure on
$\ball_X(n)$, so we write $\dc_f(G)$ to denote $\dc(G,\{\nu_n\})$, where $\nu_n$ is the
uniform measure on $\ball_f(n)$; that is,
 $$
\dc_f(G) =\limsup_{n\to \infty} \frac{|\{(u,v)\in (\ball_f(n))^2 \ : \ uv=vu \}|}{|
\ball_f(n)|^2}.
 $$
We have to think of $\nu_n$ as a small perturbation of $\mu_n$, and
equation~\eqref{estimate} will allow us to relate $\dc_X(G)$ with $\dc_f(G)$.

\begin{prop}\label{prop:est}
Let $G$ be a finitely generated polynomially growing group, and let $X$ be a finite
generating set for $G$. Let $f\colon G\to \N$ be an estimate of the $X$-metric. Then,
 $$
\dc_X(G)>0 \quad \Longleftrightarrow \quad \dc_f(G)>0.
 $$
\end{prop}

\begin{proof}
Take constants $C,D,K>0$ such that $Cn^d\leqslant |\ball_X(n)|\leqslant Dn^d$ for all
$n\geqslant 1$ (where $d\geq 1$ is the growth degree of $G$), and $f(g)/K\leqslant
|g|_X \leqslant Kf(g)$ for all $g\in G$. Then,
 $$
\ball_X(\lfloor n/K\rfloor)\subseteq \ball_f(n)\subseteq \ball_X(Kn),
 $$
and so,
 $$
|\{(u, v) \in (\ball_f(n))^2 \ : \  uv = vu\}|\leq |\{(u, v) \in (\ball_X(Kn))^2 \ : \  uv =
vu\}|.
 $$
Then, for all $n\geqslant 0$ we have that
\begin{align}\label{eq:f-est}
\frac{|\{(u, v) \in (\ball_f(n))^2 \ : \  uv = vu\}|}{|\ball_f(n)|^2}
\frac{|\ball_f(n)|^2}{|\ball_X(Kn)|^2} & =\frac{|\{(u, v) \in (\ball_f(n))^2 \ : \  uv=vu
\}|}{|\ball_X(Kn)|^2} \nonumber \\ & \leqslant \frac{|\{(u, v) \in  (\ball_X(Kn))^2 \ : \
uv=vu \}|}{|\ball_X(Kn)|^2}.
\end{align}
Also, for $n\gg 0$,
 $$
\frac{|\ball_f(n)|}{|\ball_X(Kn)|} \geqslant \frac{|\ball_X (\lfloor n/K\rfloor
)|}{|\ball_X(Kn)|} \geqslant \frac{C \lfloor n/K \rfloor^d}{D(Kn)^d}\geqslant
\frac{C(2n/(3K))^d}{DK^dn^d}=\frac{2^dC}{3^dDK^{2d}}>0,
 $$
where the last inequality works because $x\geqslant 3\Rightarrow \lfloor x\rfloor
\geqslant 2x/3$.

Now, taking $\limsup$ in~\eqref{eq:f-est}, we obtain that $\dc_f(G)\cdot \e \leqslant
\dc_X(G)$, for some $\e>0$. Hence, $\dc_X(G)=0$ implies $\dc_f(G)=0$.

A symmetric computation yields the other implication.
\end{proof}

\begin{cor}
Let $G$ be a finitely generated polynomially growing group. The positivity (or nullity) of the value $\dc_X(G)$ does not depend on the finite generating set $X$ for $G$. $\Box$
\end{cor}

Now consider a group $G=\langle X\rangle$ and a non-distorted subgroup $\langle Y\rangle=H\leqslant G$ (this is always the case, for example, when $H$ is of finite index in $G$). As noted above, the function $f\colon H\to \N$ given by $f(h)=|h|_X$ (i.e., the metric on $H$ induced by the $X$-metric on $G$) is an estimate of the $Y$-metric in $H$. Hence, denoting the number $\dc_f(H)$ by $\dc_X(H)$, Proposition~\ref{prop:est} tells us that

\begin{cor}\label{rem:subest}
Let $G$ be a finitely generated polynomially growing group, let $H\leqslant G$ be a non-distorted subgroup, and take finite generating sets $G=\langle X\rangle$ and $H=\langle Y\rangle$. Then,
 $$
\dc_X(H)>0 \quad \Longleftrightarrow \quad \dc_Y(H)>0.
 $$
\end{cor}

Using the following result from Burillo--Ventura (see~\cite[Propositions 2.3 and
2.4]{BV}) we can deduce an inequality about finite index subgroups in polynomially
growing groups, which will be useful later.

\begin{prop}[Burillo--Ventura~\cite{BV}]\label{prop:bv}
Let $G$ be a finitely generated group with subexponential growth, and let $X$ be a
finite generating set for $G$. For every finite index subgroup $H\leqslant G$ and every
$g\in G$, we have
 $$
\lim_{n\to \infty} \frac{|gH\cap \ball_X(n)|}{|\ball_X(n)|}=\lim_{n\to \infty} \frac{|Hg\cap
\ball_X(n)|}{|\ball_X(n)|}= \frac{1}{[G:H]}.
 $$
\end{prop}

\begin{rem}\label{rem:freegp}
We recall the observation from~\cite[Example~1.5]{BV} saying that, in
Proposition~\ref{prop:bv}, the subexponential growth condition is necessary. If
$\mathbb{F}_p$ is the free group in $p$ generators, $X$ is a free basis, and $N$ is the
subgroup of index two consisting of the set of elements of even length, it is easy to
see that $\limsup_{n\to \infty} |N\cap \ball_X(n)|/|\ball_X(n)|
=(2p-2)(2p-1)/((2p-1)^2-1)$. This value is different from $1/2$ and, even worse, the
limit does not exists (and, probably, different values can be obtained working with
other sets of generators). This easy example illustrates the extra difficulties of the
exponentially growing case.
\end{rem}

\begin{prop}\label{cor:vab}
Let $G$ be a finitely generated polynomially growing group, $H\leqslant G$ a finite
index subgroup, and take finite generating sets $G=\langle X\rangle$ and $H=\langle
Y\rangle$. Then,
 $$
\dc_X(G)\geq \dfrac{\dc_X(H)}{[G:H]^2}.
 $$
In particular,  $\dc_Y(H)>0$ implies $\dc_X(G)>0$.
\end{prop}

\begin{proof}
Notice that $|\{(u,v)\in (\ball_X(n))^2\ : \  uv=vu\}|\geq |\{(u,v)\in
(H\cap\ball_X(n))^2\ : \  uv=vu\}|$ and therefore
\begin{equation}\label{eq:fin1}
\dfrac{|\{(u,v)\in (\ball_X(n))^2\ : \  uv=vu\} |}{|\ball_X(n)|^2} \geq
\dfrac{|\{(u,v)\in (H\cap\ball_X(n))^2\ : \  uv=vu\}|}{|H\cap \ball_X(n)|^2} \cdot
\dfrac{|H\cap \ball_X(n)|^2}{|\ball_X(n)|^2}.
\end{equation}
Now, given $\e>0$, Proposition~\ref{prop:bv} tells us that
 $$
\frac{|H\cap \ball_X(n)|}{|\ball_X(n)|}\geq \frac{1}{[G:H]}-\e,
 $$
for $n\gg 0$; therefore, from~\eqref{eq:fin1}, we deduce that
 $$
\dfrac{|\{(u,v)\in (\ball_X(n))^2\ : \  uv=vu\} |}{|\ball_X(n)|^2}\geq \dfrac{|\{(u,v)\in
(H\cap\ball_X(n))^2\ : \  uv=vu\}|}{|H\cap \ball_X(n)|^2} \left(\dfrac{1}{[G:H]}-\e\right)^2,
 $$
for $n\gg 0$. Taking $\limsup$, we obtain $\dc_X(G)\geqslant \dc_X(H)(1/[G:H]-\e)^2$.
Since this is true for every $\e>0$, we deduce the desired inequality
$\dc_X(G)\geqslant \dc_X(H)/[G:H]^2$.

Hence, by Corollary~\ref{rem:subest}, $\dc_Y(H)>0 \, \Rightarrow \, \dc_X(H)>0 \,
\Rightarrow \, \dc_X(G)>0$.
\end{proof}

\section{Groups of subexponential growth and the proof of Theorem~\ref{thm:main}}\label{3}

In the finite realm, the degree of commutativity behaves well with respect to normal
subgroups and quotients. The first statement in this direction is the following one due
to Gallagher (and meaning that a group $G$ is \emph{at most} as abelian as any normal
subgroup and any quotient of itself).

\begin{lemma}[Gallagher, \cite{Ga}]\label{lem:formula-finite}
Let $G$ be a finite group, and $N\unlhd G$ a normal subgroup. Then,
 $$
\dc(G)\leqslant \dc(N)\cdot \dc(G/N).
 $$
\end{lemma}

We develop now a simpler version of this result for infinite groups, which will be
enough to prove Theorem~\ref{thm:main}.

\begin{prop}\label{prop:rf}
Let $G$ be a finitely generated subexponentially growing group, and let $X$ be a finite
generating system for $G$. Then, for any finite quotient $G/N$, we have
 $$
\dc_X(G) \leqslant \dc(G/N).
 $$
\end{prop}

\begin{proof}
Let $N\unlhd G$ be a normal subgroup of $G$ of index, say, $[G:N]=d$. We will show that
$\dc_X(G)\leqslant \dc(G/N)$.

By Proposition~\ref{prop:bv}, for every $g\in G$ we have
 $$
\lim_{n\to \infty} \frac{| gN\cap \ball_X(n)|}{|\ball_X(n)|}=\frac{1}{d},
 $$
independently from $X$ and $g$; additionally, this is a real limit, not just a
$\limsup$. Since there are finitely many classes modulo $N$, the previous limit is
\emph{uniform} on $g$, i.e., for every $\e>0$ there exists $n_0$ such that, for every
$n\geqslant n_0$ and \emph{all} $g\in G$,
 \begin{equation}\label{unif}
\left(\frac{1}{d}-\varepsilon \right)|\ball_X(n)|\leqslant |gN\cap \ball_X(n)|\leqslant
\left(\frac{1}{d}+\varepsilon \right)|\ball_X(n)|.
 \end{equation}

Now, suppose $\dc_X(G)>\dc(G/N)$ and let us find a contradiction. By definition, this
means that there exist $\delta>0$ for which
 $$
\frac{|\{(u,v)\in (\ball_X(n))^2\ : \  uv=vu \}|}{|\ball_X(n)|^2}>\dc(G/N)+\delta
 $$
for infinitely many $n$'s. In view of this $\delta$, take $\varepsilon>0$ small enough
so that $\e d(2+\e d)\leqslant \delta$, and we have~\eqref{unif} for all but finitely
many $n$'s. Combining both assertions, there is a big enough $n$ such that
\begin{eqnarray*}
\dc(G/N)+\delta & < & \frac{|\{ (u,v)\in (\ball_X(n))^2 \ : \  uv=vu \} |}{|\ball_X(n)|^2} \\
& \leqslant & \frac{1}{|\ball_X(n)|^2}\, |\{(\overline{u},\overline{v}) \in (G/N)^2 \ : \
\overline{u}\,\overline{v}=\overline{v}\,\overline{u}\}|\, \left(\frac{1}{d}+\e
\right)^2|\ball_X(n)|^2 \\ & = & \dfrac{|\{(\overline{u},\overline{v}) \in
(G/N)^2 \ : \  \overline{u}\,\overline{v}=\overline{v}\,\overline{u}\}|}{d^2}(1+\e d)^2
\\ & \leqslant & \dfrac{|\{(\overline{u},\overline{v}) \in
(G/N)^2 \ : \  \overline{u}\,\overline{v}=\overline{v}\,\overline{u}\}|}{d^2} +2\e d+\e^2d^2
\\ & = & \dc(G/N) +\e d(2+\e d),
\end{eqnarray*}
where the second inequality comes from the facts that $uv=_G vu$ implies
$\overline{u}\,\overline{v}=_{G/N}\overline{v}\,\overline{u}$ and that, according
to~\eqref{unif}, every class in $G/N$ has, at most, $\left(1/d+\varepsilon
\right)|\ball_X(n)|$ representatives in $\ball_X(n)$. We deduce $\delta <\e d(2+\e d)$,
which is a contradiction.
\end{proof}

\begin{proof}[Proof of Theorem~\ref{thm:main}]
Recall that a group $G$ is {\it residually finite} if, for every non-trivial element
$1\neq g\in G$, there is a finite quotient of $G$ where the image of $g$ is
non-trivial. Equivalently, for every $g\neq 1$, there exists a finite index normal
subgroup $N\unlhd G$ such that $g\not\in N$. When $G$ is finitely generated, we can
always additionally take this subgroup to be \emph{characteristic in $G$} (i.e.,
invariant under every automorphism of $G$): take, for instance, $K\leqslant G$ to be the intersection of all subgroups of $G$ whose index is the same as that of $N$ (there are finitely many, so $K\leqslant N$ is still of finite index in $G$).

Assertion~(ii) follows directly from Proposition~\ref{prop:rf}: one implication is
trivial; for the other, if $G$ satisfies $\dc_X(G)>5/8$ then all its finite quotients
do and so, by Gustafson's Theorem~\ref{thm:5/8}, they all are abelian. This already
implies that $G$ is itself abelian: if $g_1,g_2\in G$ do not commute, then $1\neq [g_1,
g_2]$ would survive in some finite quotient $Q=G/N$, contradicting the fact that such
$Q$ is abelian.

For part (i), assume $G$ is virtually abelian and let $H\leqslant G$ be an abelian
subgroup of finite index. Then $G$ has polynomial growth, and we can apply
Proposition~\ref{cor:vab}: since $\dc_Y(H)=1>0$ we deduce $\dc_X(G)>0$ (where $Y$ is
any finite generating set for $H$).

Conversely, assume that $G$ is not virtually abelian, and let us prove that
$\dc_X(G)=0$. Knowing that $G$ is finitely generated, residually finite, and not
virtually abelian, we can choose two non-commuting elements $g_1,g_2\in G$ and a
characteristic subgroup $K_1$ of finite index in $G$ such that $[g_1, g_2]\not\in K_1$ (hence, $G/K_1$ is non-abelian and so, $\dc(G/K_1) \leq 5/8$). Clearly, these three properties go to finite index subgroups and so we can repeat the construction and get a descending sequence of subgroups,
 $$
\cdots \unlhd K_i \unlhd K_{i-1} \unlhd \cdots \unlhd K_2 \unlhd K_1 \unlhd K_0=G,
 $$
each characteristic and of finite index in the previous one, and such that
$\dc(K_i/K_{i+1}) \leq 5/8$. Then, for every $i$, $[G:K_i]<\infty$ and $K_i$ is
characteristic (and so normal) in $G$. Now, since $(G/K_{i})/(K_{i-1}/K_{i})
=G/K_{i-1}$, Lemma~\ref{lem:formula-finite} tells us that
 $$
\dc(G/K_{i})\leqslant \dc(K_{i-1}/K_{i})\cdot \dc(G/K_{i-1}) \leqslant \frac{5}{8}\dc(G/K_{i-1}),
 $$
for every $i\geqslant 1$. By induction, $\dc(G/K_i)\leqslant (5/8)^i$ and then, by
Proposition~\ref{prop:rf}, we get
 $$
\dc_X(G)\leqslant \dc(G/K_i)\leqslant (5/8)^i.
 $$
Since this is true for every $i\geqslant 1$, we conclude that $\dc_X(G)=0$.
\end{proof}

\section{Hyperbolic groups}\label{4}

We now give another criterion to show that certain groups have degree of commutativity
equal to zero. It will apply to (many) exponentially growing groups, not contained in
the results from the previous sections.

\begin{lemma}\label{lem:exp}
Let $G$ be a finitely generated group, and let $X$ be a finite generating system for
$G$. Suppose that there exists a subset $\cN \subseteq G$ satisfying the following
conditions:
\begin{enumerate}
\item[(i)] $\cN$ is \emph{$X$-negligible}, i.e., $\lim_{n\to \infty} \dfrac{|\cN \cap
    \ball_X(n)|}{|\ball_X(n)|}=0$;
\item[(ii)] $\lim_{n\to \infty}\dfrac{|C(g)\cap \ball_X(n)|}{|\ball_X(n)|}=0$ uniformly in $g\in G \setminus \cN$.
\end{enumerate}
Then, $\dc_X(G)=0$.
\end{lemma}

\begin{proof}

By the two hypothesis, given $\e >0$ there exists $n_0$ and $n_1$ such that, for $n\geqslant \max \{ n_0, n_1\}$,  we have $|\cN \cap \ball_X(n)| <\frac{\e}{2}|\ball_X(n)|$ and $|C(g)\cap \ball_X(n)|<\frac{\e}{2}|\ball_X(n)|$. Hence,
\begin{align*}
|\{(u,v)\in (\ball_X(n))^2\mid uv=vu\}| & = \sum\limits_{u\in \ball_X(n)} |C(u)\cap \ball_X(n)| \\ & = \sum\limits_{u\in \ball_X(n)\setminus \cN} |C(u)\cap \ball_X(n)|+\sum\limits_{u\in \cN\cap \ball_X(n)} |C(u)\cap \ball_X(n)| \\ & \leqslant\sum \limits_{u\in \ball_X(n)\setminus \cN}   \frac{\e}{2}|\ball_X(n)| +\sum\limits_{u\in \cN\cap \ball_X(n)} |\ball_X(n)| \\ & \leqslant \frac{\e}{2}|\ball_X(n)|^2 + \frac{\e}{2}|\ball_X(n)|^2  \\ & =\e |\ball_X(n)|^2.
\end{align*}
Therefore, $\dc_X(G)=0$ (with existence of the real limit, not just a $\limsup$).
\end{proof}


\begin{proof}[Proof of Theorem \ref{thm:nafs}]
Fix any finite generating set $X$ for $G$. Since $G$ is non-elementary hyperbolic it
has exponential growth, i.e., $\lim_{n\to \infty} |\ball_X(n+1)|/|\ball_X(n)|
=\lambda>1$.

For the whole proof, fix $\e_0>0$ so that $\lambda -\e_0 >1$.

The above limit means that $\lambda -\e_0 <|\ball_X(n+1)|/|\ball_X(n)|<\lambda +\e_0$,
for $n\gg 0$. Hence, using induction, there exist constants $C_1$ and $C_2$ such that
$C_1 (\lambda-\e_0 )^n\leqslant |\ball_X(n)| \leqslant C_2 (\lambda+\e_0 )^n$ for all $n\in \N$. By induction on $m$, it also follows that $|\ball_X(n)|/|\ball_X(n+m)| \leqslant 1/(\lambda -\e_0 )^m$, for $n\gg 0$ and for all $m\geqslant 1$.

Let $\cN$ be the set of torsion elements in $G$; we shall see that the two conditions
in Lemma~\ref{lem:exp} are satisfied with respect to this set. To start, it was proved
by Dani, see~\cite[Theorem 1.1]{D} that there exist constants $D_1$ and $D_2$ such that
$|\cN \cap \ball_X(n)|\leqslant D_1 |\ball_X(\lceil \frac n2 \rceil +D_2)|$. Therefore,
for $n\gg 0$ we have
 $$
\frac{|\cN \cap \ball_X(n)|}{|\ball_X(n)|}\leqslant \frac{D_1 |\ball_X(\lceil \frac n2
\rceil +D_2)|}{|\ball_X(n)|} \leqslant \frac{D_1}{(\lambda-\e_0 )^{n-\lceil\frac n2 \rceil -D_2}} =\frac{D_1}{(\lambda-\e_0 )^{\lfloor \frac n2 \rfloor -D_2}}
 $$
so, $\cN$ is negligible and condition~(i) in Lemma~\ref{lem:exp} is satisfied.

Before proving~(ii), let us recall a few well-known facts about hyperbolic groups.  For $g\in G$, $\tau(g)=\lim_{n\to \infty} |g^n|_X/n$ denotes the \emph{stable translation length} of $g$. Since $|g^{n+m}|_X \leqslant |g^n|_X +|g^m|_X$, Fekete's lemma gives that $\tau(g)=\inf_{n\in \N} \{ |g^n|_X /n \}$ and so, $|n|\tau(g)\leqslant |g^n|_X$ for all $n\in \Z$. The hyperbolicity of $G$ implies that these translation lengths are discrete (see~\cite[III.$\Gamma$.3.17]{BH}); in particular, there is a positive integer $p$ such that $\tau(g)\geqslant 1/p$ for all $g\in G\setminus \cN$. Also, centralizers of elements of infinite order in hyperbolic groups are virtually cyclic, see~\cite[Corollary~III.$\Gamma$.3.10]{BH}, and there is a bound $M>0$ on the size of finite subgroups of $G$ (depending only on the hyperbolicity constant of $G$), see~\cite[Theorem~III.$\Gamma$.3.2]{BH}.

Let now $C=\langle g\rangle$ be an infinite cyclic subgroup of $G$. Then, $g^k \in C \cap \ball_X(n)$ implies that $|k|/p \leq |k|\tau(g) \leq |g^k|_X \leq n$, and hence $|k|\leq pn$; therefore,
 $$
|C \cap \ball_X(n)| \leqslant 2pn+1.
 $$
Furthermore, we can also deduce that, for any $x\in G$, the coset of $C$ containing $x$ also grows linearly: this is because if $xC \cap \ball_X(n)$ is non-empty, it will contain some $w$ of length at most $n$, $w^{-1}(xC \cap \ball_X(n)) \subseteq C \cap \ball_X(2n)$ and hence,
 $$
|xC \cap \ball_X(n)| =|w^{-1}(xC \cap \ball_X(n))| \leqslant |C\cap \ball_X(2n)| \leqslant 4pn+1.
 $$

To check condition~(ii) from Lemma~\ref{lem:exp}, take $g\in G\setminus \cN$. The centralizer $C_G(g)$ is virtually cyclic and, by a classical result, it is also of type finite-by-$\mathbb{Z}$ or finite-by-$\mathbb{D}_\infty$, where $\mathbb{D}_\infty$ is the infinite dihedral group (see~\cite[Lemma~4.1]{Wall}). Passing to a subgroup of index two $H\leqslant C_G(g)$ if necessary, we have a short exact sequence $1\to K\to H\to \Z\to 1$, with $K (\leqslant H\leqslant C_G(g)\leqslant G)$ finite. Since this sequence splits, $\Z$ is a subgroup of $H$ of index $|K|\leqslant M$ and so, $C_G(g)$ has a subgroup of index at most $2M$ which is infinite cyclic. Putting these results together, we get that 
 $$
\frac{|C(g)\cap \ball_X(n)|}{|\ball_X(n)|}\leqslant \frac{2M(4pn+1)}{C_1(\lambda-\e_0)^n} \longrightarrow 0
 $$
uniformly on $g\in G \setminus \cN$, when $n\to \infty$. So, condition~(ii) in Lemma~\ref{lem:exp} is satisfied.

Therefore, from Lemma~\ref{lem:exp} we conclude that $\dc_X(G)=0$.
\end{proof}

\noindent{\textbf{{Acknowledgments}}} The authors acknowledge partial support through
Spanish grant number MTM2014-54896.



\begin{thebibliography}{1}

\bibitem{Bass} H. Bass, {\em The degree of polynomial growth of finitely generated
    nilpotent groups}, Proc. London Math. Soc. \textbf{25} (1972), 603--614.

\bibitem{BH} M. Bridson and A. Haefliger, {\em Metric spaces of non-positive
    curvature}, Grundlehren der Mathematischen Wissenschaften [Fundamental Principles
    of Mathematical Sciences] \textbf{319}, Springer-Verlag, Berlin, 1999. xxii+643 pp.

\bibitem{BV} J. Burillo and E. Ventura, {\em Counting primitive elements in free
    groups}, Geom. Dedicata \textbf{93} (2002), 143--162.




\bibitem{D} P.\ Dani, {\em Genericity of infinite-order elements in hyperbolic groups}, Preprint.


\bibitem{Ga} P.X. Gallagher, {\em The number of conjugacy classes in a finite group},
    Math Z. \textbf{118} (1970), 175--179.

\bibitem{Gromov-polyn_gr} M. Gromov, {\em Groups of polynomial growth and expanding maps}, Inst. Hautes \'Etudes Sci. Publ. Math. \textbf{53} (1981), 53--73.

\bibitem{G} W.H. Gustafson, {\em What is the probability that two group elements
    commute?}, The American Mathematical Monthly \textbf{80}(9) (1973), 1031--1034.

\bibitem{Gr} R.I. Grigorchuk, \emph{On the Milnor problem of group growth} (Russian),
    Dokl. Akad. Nauk SSSR \textbf{271}(1) (1983), 30-–33.


\bibitem{He} P. Hegarty, \emph{Limit points in the range of the commuting probability
    function on finite groups}, J. Group Theory \textbf{16}(2) (2013), 235–247.

%

\bibitem{Po} V. Ponomarenko and N. Selinski, \emph{Two semigroup elements can commute
    with any positive rational probability}, College Mathematics Journal \textbf{43}(4) (2012), 334-336.

\bibitem{Wall} C.T.C. Wall, \emph{Poincare complexes: I}, The Annals of Mathematics \textbf{86}(2) (1967), 213--245.





\end{thebibliography}
\end{document}